\documentclass{article}
\usepackage[utf8]{inputenc}
\usepackage[OT1]{fontenc}
\def\affaddr{}
\newcommand{\email}[1]{{\small\texttt{#1}}}

\usepackage{amsthm}
\usepackage{amsmath}
\usepackage{amsfonts}
\usepackage{bm} 
\usepackage{mathrsfs} 
\usepackage{multicol}
\usepackage[naturalnames]{hyperref}

\theoremstyle{definition}
\newtheorem{definition}{Definition}[section]

\newtheorem{remark}[definition]{Remark}
\newtheorem{example}[definition]{Example}

\theoremstyle{plain}
\newtheorem{lemma}[definition]{Lemma}
\newtheorem{proposition}[definition]{Proposition}
\newtheorem{theorem}[definition]{Theorem}
\newtheorem{corollary}[definition]{Corollary}

\newcommand{\bxi}{{{\mathbf{\xi}}}}
\newcommand{\bmu}{{{\mathbf{\mu}}}} 
\newcommand{\bnu}{{{\mathbf{\nu}}}} 
\newcommand{\bpartial}{{\mbox{\boldmath$\partial$}}}
\newcommand{\balpha}{{\alpha}}
\newcommand{\bbeta}{{\beta}}
\newcommand{\bgamma}{{\gamma}}
\newcommand{\bLambda}{\mathbf{\Lambda}}
\newcommand{\pp}[1]{d_{#1}} 

\def\nil{o}
\def\rank{{\rm rank}}
\def\mM{{\tt M}}

\def\xb{\mathbf{x}}
\def\fb{\mathbf{f}}

\def\bx{{\bf x}}
\def\Cc{{\mathcal C}}
 
\def\Nc{{\mathcal N}} 
\def\f{{\bf f}}
\def\bz{{\bf z}}
\def\bM{{\bf M}}
\def\e{{\bf e}}

\def\K{{\mathbb C}}
\def\KK{{\mathbb K}}

\def\I{{\mathcal I}}

\def\QQ{{Q}}

\def\C{{\mathbb C}}

\def\N{{\mathbb N}}
\def\Q{{\mathbb Q}}

\def\ord{\mathrm{ord}}
\def\sp{{\rm{span}}}
\def\DDD{{\mathscr D}}
\def\m{{\mathfrak m}}
\def\mult{{\delta}}
\def\bmu{{\mu}}

\begin{document}


\title{Certifying isolated singular points and their multiplicity structure}

\author{
\hspace{-0.75cm}\begin{tabular}{ccc}
 \begin{minipage}{5.5cm}
Jonathan D. Hauenstein\thanks{Research partly supported by DARPA YFA, NSF grant 
ACI-1460032, and Sloan Research Fellowship.}\\
      \affaddr{Department of Applied and Computational Mathematics and Statistics,}\\
       \affaddr{University of Notre Dame}\\
      \affaddr{Notre Dame, IN, 46556, USA}\\
       \email{hauenstein@nd.edu}
\end{minipage}
& 
\begin{minipage}{4.5cm}
Bernard Mourrain\\
\affaddr{Inria Sophia Antipolis
}\\
\affaddr{2004 route des Lucioles, B.P. 93,}\\
\affaddr{06902 Sophia Antipolis,}\\
\affaddr{Cedex France}\\
       \email{bernard.mourrain@inria.fr }
\end{minipage} 
& 
\begin{minipage}{6cm}
Agnes Szanto\thanks{Research partly supported by NSF grant CCF-1217557.}\\
       \affaddr{Department of Mathematics, }\\
       \affaddr{North~Carolina~State~University}\\
       \affaddr{Campus Box 8205,}\\
       \affaddr{Raleigh, NC, 27965, USA.}\\
       \email{aszanto@ncsu.edu}\\
\end{minipage}
\end{tabular}
}

\date{10 July 2015}
\maketitle

\begin{abstract}
  This paper presents two new constructions related to singular solutions of polynomial systems. 
  The first is a new deflation method for  
  an isolated singular root. This construction uses a single linear
  differential form defined from the Jacobian matrix of the input, and
  defines the deflated system by applying this differential form to
  the original system.  The advantages of this new deflation is that it
  does not introduce new variables and the increase in the number of
  equations is linear instead of the quadratic increase of previous
  methods. The second construction gives the coefficients of
  the so-called inverse system or dual basis, which defines the
  multiplicity structure at the singular root.  We present a system of
  equations in the original variables plus a relatively small number
  of new variables.  We show that the roots of this new system
  include the original singular root but now with multiplicity
  one, and the new variables uniquely determine the multiplicity structure.  
  Both constructions are ``exact'' in that they permit one to 
  treat all conjugate roots simultaneously and can be used 
  in certification procedures for singular roots and 
  their multiplicity structure with respect to 
  an exact rational polynomial system.
\end{abstract}

\section{Introduction}

Our motivation for this work is twofold. On one hand, in a recent paper \cite{AkHaSz2014},  two of the co-authors of the present paper studied a certification method
for approximate roots of exact overdetermined and singular polynomial systems, and wanted to extend the method to certify the multiplicity structure at the root as well. Since all these problems are ill-posed, in \cite{AkHaSz2014} a hybrid symbolic-numeric approach was proposed, that included the exact computation of a square polynomial system that had the original root with multiplicity one. In certifying singular roots, this exact square system was obtain from a deflation technique that added subdeterminants of the Jacobian matrix to the system iteratively. However, the multiplicity structure is destroyed by this deflation technique, that is why  it remained an open question how to certify the multiplicity structure of singular roots of exact polynomial systems.

Our second motivation was to find a method that  simultaneously refines the accuracy of a singular root and the parameters describing the multiplicity structure at the root. In all previous numerical approaches that approximate these parameters, they apply numerical linear algebra to solve a linear system with coefficients depending on the approximation of the coordinates of the singular root. Thus  the local convergence rate of the parameters was slowed   from the quadratic convergence of Newton's iteration applied to the singular roots. We were interested if the parameters describing the multiplicity structure can be simultaniously approximated with the coordinates of the singular root using Newton's iteration. 

In the present paper we first give a new improved version of the deflation method that can be used in the certification algorithm of \cite{AkHaSz2014}, reducing the number of added equations at each deflation iteration from quadratic to linear. We prove that applying a single linear differential form to the input system, corresponding to a generic kernel element of the Jacobian matrix, already reduced both the multiplicity and the depth of the singular root. Secondly, we give a description of 
the multiplicity structure using a  polynomial number of parameters, and express these parameters together with the coordinates of the singular point as the roots of a multivariate polynomial system. We prove that this new polynomial system has a root corresponding to the singular root but now with multiplicity one, and the new added coordinates describe the multiplicity structure. Thus this second approach completely deflates the system in one step. The number of equations and variables in the second construction depends polynomially on the number of variables and  equations of the input system and the multiplicity of the singular root. Both constructions
  are exact in the sense that approximations of the coordinates of
  the singular point are only used to detect numerically non-singular
  submatrices, and not in the coefficients of the constructed polynomial systems.

\textbf{Related work.}

The treatment of singular roots is a critical issue for numerical analysis and there is a
huge literature on methods which transform the problem into a new one
for which Newton-type methods converge quadratically to the root.
 
Deflation techniques which add new equations in order to
reduce the multiplicity have already been considered in
\cite{Ojika1983463}, \cite{Ojika1987199}:
By triangulating the Jacobian matrix at the (approximate) root,
new minors of the polynomial Jacobian matrix are added to the initial
system in order to reduce the multiplicity of the singular solution.

A similar approach is used in \cite{HauWam13} and \cite{GiuYak13}, where a maximal
invertible block of the Jacobian matrix at the (approximate) root is
computed and minors of the polynomial Jacobian matrix are added to the
initial system. In \cite{GiuYak13}, an additional step  is
considered where the first derivatives of the input polynomials are
added when the Jacobian matrix at the root vanishes.

These constructions are repeated until a system with a simple root is obtained.

In these methods, at each step, the number of added equations is $(n-r)\times (m-r)$,
where $n$ is number of variables, $m$ is the number of equations and $r$
is the rank  of the Jacobian at the root.

In~\cite{Lecerf02}, a triangular presentation of the ideal in a
good position and derivations with respect to the leading variables are used
to iteratively reduce the multiplicity. This process is applied for p-adic
lifting with exact computation.  


In other approaches, new variables and new equations are introduced simultaneously.
In \cite{YAMAMOTO:1984-03-31}, new variables are introduced to
describe some perturbations of the initial equations and some differentials which
vanish at the singular points.
This approach is also used in \cite{LiZhi2014}, where it is shown that
this iterated deflation process yields a system with a simple root.

In \cite{mantzaflaris:inria-00556021}, 
perturbation variables  are 
also introduced in relation with the inverse system of the singular point
to obtain directly a deflated system with a simple root.
The perturbation is constructed from a monomial basis of the local
algebra at the multiple root.

In~\cite{lvz06,lvz08}, only variables for the differentials of the initial
system are introduced.
The analysis of this deflation is improved in \cite{DaytonLiZeng11},
where it is shown that the number of steps is bounded by the order
of the inverse system.
This type of deflation is also used in \cite{LiZhi2013}, for the special case
where the Jacobian matrix at the multiple root has rank $n-1$ (case of
breath one).

In these methods, at each step, the number of variables is at least doubled and new
equations are introduced, which are linear in these new variables.

The mentioned deflation techniques usually breaks the structure of the local
ring at the singular point. The first method to compute the inverse system describing this structure is due to
F.S. Macaulay \cite{mac1916}
and known as the dialytic method.
More recent algorithms for the construction of inverse systems are described
e.g. in \cite{Marinari:1995:GDM:220346.220368}, reducing the size of the
intermediate linear systems (and exploited in
\cite{Stetter:1996:AZC:236869.236919}). 
In~\cite{zeng05}, the dialytic method is used, and they analyze the relationship of deflation some methods to the inverse system.
It had been further improved in \cite{Mourrain97} and more recently in 
\cite{mantzaflaris:inria-00556021},
using an integration method. This technique reduces significantly the
cost of computing the inverse system, since it relies on the solution
of linear system related to the inverse system truncated in some
degree and not on the number of monomials in this degree.
Multiplication matrices corresponding to systems with singular roots were studied in \cite{Moller95,Corless97}.

The computation of inverse systems has been used to approximate a
multiple root.
In~\cite{Pope2009606}, a minimization approach is used to reduce the value of
the equations and their derivatives at the approximate root, assuming a basis
of the inverse system is known. 
In 
\cite{WuZhi2011}, the inverse system is constructed via
Macaulay's method; tables of multiplications are deduced and their
eigenvalues are used to improve the approximated root. They 
show that the convergence is quadratic at the multiple root. In \cite{LiZhi12}
they show that in the breadth one case the parameters needed to describe the 
inverse system is small, and use it to compute the singular roots in  \cite{LiZhi12b}.
In \cite{mantzaflaris:inria-00556021}, the inverse system is used to
transform the singular root into a simple root of an augmented system.




\textbf{Contributions.} 

We propose a new deflation method for polynomial systems with
isolated singular points, which does introduce new
parameters. At each step, a single differential of the system is
considered based on the analysis of the Jacobian at the singular
point. A linear number of new
equations is added instead of the quadratic increases of the previous
deflations.
The deflated system does not involved any approximate coefficients and
can therefore be used in certification methods as in \cite{AkHaSz2014}.

To approximate efficiently both the singular point and its inverse
system, we propose a new deflation, which involves a small number of
new variables compared to other approaches which rely on Macaulay
matrices. It is based on a new characterization of the isolated
singular point together with its multiplicity structure.  The deflated
polynomial system exploits the nilpotent and commutation properties of
the multiplication matrices in the local algebra of the singular
point.  We prove that it has a simple root which yields the root and
the coefficients of the inverse system at this singular point. Due to
the upper triangular form of the multiplication matrices in a
convenient basis of the local algebra, the number of new parameters
introduced in this deflation is less than ${1\over 2} n (\mult -1)\mult $
where $n$ is the number of variables and $\mult$ the multiplicity of the
singular point.  The parameters involved in the deflated system are
determined from the analysis of an approximation of the singular
point. Nevertheless, the deflated system does not involve any
approximate coefficients and thus it can also be used in certification
techniques as \cite{AkHaSz2014}.

  In this paper we present two new constructions. The first one is a
  new deflation method for a system of polynomials with an isolated
  singular root. The new construction uses a single linear
  differential form defined from the Jacobian matrix of the input, and
  defines the deflated system by applying this differential form to
  the original system.  We prove that the resulting deflated system
  has strictly lower multiplicity and depth at the singular point than
  the original one. The advantage of this new deflation is that it
  does not introduce new variables, and the increase in the number of
  equations is linear, instead of the quadratic increase of previous
  deflation methods. The second construction gives the coefficients of
  the so called inverse system or dual basis, which defines the
  multiplicity structure at the singular root. The novelty of our
  construction is that we show that the nilpotent and commutation
  properties of the multiplication matrices define smoothly the
  singular points and its inverse system.  We give a system of
  equations in the original variables plus a relatively small number
  of new variables, and prove that the roots of this new system
  correspond to the original multiple roots but now with multiplicity
  one, and they uniquely determine the multiplicity structure.  The
  number of unknowns used to describe the multiplicity structure is
  significantly smaller, compared to the direct computation of the
  dual bases from the so called Macaulay matrices.  Both constructions
  are ``exact" in the sense that approximations of the coordinates of
  the singular point are only used to detect numerically non-singular
  submatrices, and not in the rest of the construction. Thus these
  constructions would allow to treat all conjugate roots
  simultaneously, as well as to apply these constructions in the
  certification of the singular roots and the multiplicity structure
  of an exact rational polynomial system.

\section{Preliminaries}

Let $\f:= (f_1, \ldots, f_N)\in \KK[\bx ]^N$ with  $\bx =(x_1, \ldots, x_n)$ for some $\KK\subset \C$ field. Let $\bxi=(\xi_1, \ldots, \xi_n)\in \C^n$ be an isolated multiple  root of $\f$. 
Let $I=\langle f_1, \ldots, f_N\rangle$, $\m_{\xi}$ be the maximal ideal at ${\xi}$ and $\QQ$ be the primary component of $I$ at $\bxi$ so that $\sqrt{\QQ}=\m_{\xi}$.\\

Consider the ring of power series $\K[[\bpartial_\xi]]:= \K[[\partial_{1,\xi}, \ldots, \partial_{n,\xi}]]$ and we use the notation 
for  $\bbeta=(\beta_1, \ldots, \beta_n)\in \N^n$  
$$\bpartial^{\bbeta}_{\xi}:=\partial_{1, \xi}^{\beta_1}\cdots \partial_{n, \xi}^{\beta_n}.$$ 
We identify $\C[[\bpartial_\xi]]$ with the dual space $\C[\bx ]^*$  by
considering $\bpartial^{\bbeta}_{\xi}$ as derivations and evaluations at $\bxi$,  defined by 
\begin{equation}\label{partial}
\bpartial^{\bbeta}_{ \xi}(p):=\bpartial^\bbeta(p)|_{\bxi} := \frac{d^{|\bbeta|  } p}{d x_1^{\beta_1}\cdots d x_n^{\beta_n}} (\bxi) \quad \text{ for } p\in \C[\bx].
\end{equation}
Hereafter, the derivations ``at $\xb$'' will be denoted
$\bpartial^{\bbeta}$ instead of $\bpartial^{\bbeta}_{\xb}$. The
derivation with respect to the variable $\partial_{i}$ is denoted
$\pp{\partial_{i}}$ $(i=1,\ldots, n)$.
Note that 
$$
 \frac{1}{\bbeta!} \bpartial^{\bbeta}_{ \xi}((\bx-\bxi)^\balpha)=\begin{cases} 1  & \text{ it } \balpha=\bbeta\\
0 & \text{ otherwise}
\end{cases},
$$
where we use the notation $\frac{1}{\bbeta!}= \frac{1}{\beta_1!\cdots \beta_n!} $.

For $p\in \K[\bx]$ and $ \Lambda\in \K[[\bpartial_\xi]]=\K[\bx]^{*}$,
let 
$$
p\cdot \Lambda: q \mapsto \Lambda (p\,q).
$$
We check that $p=(x_i-\xi_i)$ acts as a derivation on
$\C[[\bpartial_\xi]]$:
$$ 
(x_{i}-\xi_{i}) \cdot \bpartial^{\bbeta}_{ \xi}= \pp{\partial_{i, \xi}} (\bpartial^{\bbeta}_{ \xi})
$$

For an ideal $I\subset \K[\bx]$, let $I^{\perp}=\{\Lambda \in
\K[[\bpartial_{\xi}]]\mid \forall p\in I, \Lambda (p)=0\}$.
The vector space $I^{\perp}$ is naturally identified with the dual
space of $\K[\bx]/I$.
We check that $I^{\perp}$ is a vector subspace of $\K[[\bpartial_{\xi}]]$,
which is stable by the derivations ${d_{\partial_{i, \xi}}}$.

\begin{lemma}\label{lem:primcomp}
If $Q$ is a $\m_{\xi}$-primary component of $I$, then $Q^{\perp}=I^{\perp}\cap\K[\bpartial_{\xi}]$.
\end{lemma}

This lemma shows that to compute $Q^{\perp}$, it suffices to compute
all polynomials of $\K[\bpartial_{\xi}]$ which are in $I^{\perp}$.
Let us denote this set $\DDD= I^{\perp}\cap\K[\bpartial_{\xi}]$. It is a
vector space stable under the derivations
$\pp{\partial_{i, \xi}}$. Its dimension is the dimension of
$Q^{\perp}$ or $\K[\bx]/Q$, that is the {\em multiplicity} of
$\xi$, denote it by $\mult_{\xi} (I)$, or simply by $\mult$  if $\xi$ and $I$ is clear from the context.

For an element $\Lambda(\bpartial_\xi) \in \K[\bpartial_\xi]$ we
define the {\em order}  ${\rm ord}(\Lambda)$ to be the maximal
$|\bbeta|$ such that $\bpartial^{\bbeta}_{ \xi}$ appears in
$\Lambda(\bpartial_\xi)$ with non-zero coefficient.  

For $t\in \N$, let $\DDD_{t}$ be the elements of $\DDD$ of order $\leq t$.
As $\DDD$ is of dimension $d$, there exists a smallest $t\geq 0$ such that
$\DDD_{t+1}= \DDD_{t}$. Let us call this smallest $t$, the {\em nil-index} of
$\DDD$ and denote it by $\nil_{\xi} (I)$, or simply by $\nil$. As $\DDD$ is stable by the derivations
$\pp{\partial_{i, \xi}}$,
we easily check that for $t\geq \nil_{\xi} (I)$, $\DDD_{t}=\DDD$ and
that $\nil_{\xi} (I)$ is the maximal degree of the elements in $\DDD$.

\section{Deflation using first differentials}\label{Sec:Deflation}

\noindent To improve the numerical approximation of a root, one usually
applies a Newton-type method to converge quadratically from
a nearby solution to the root of the system, provided it is simple.
In the case of multiple roots, deflation techniques are employed to
transform the system into another one which has an equivalent root
with a smaller multiplicity or even with multiplicity one.

We describe here a construction, using differentials of order one,
which leads to a system with a simple root. This construction improves
the constructions in \cite{lvz06,DaytonLiZeng11} 
since no new variables are added. 
It also improves the constructions presented in 
\cite{HauWam13} and the ``kerneling'' method of \cite{GiuYak13}
by adding a smaller number of equations at each deflation step.
Note that, in \cite{GiuYak13}, there are smart preprocessing and postprocessing
steps which could be utilized in combination with our method.  In the preprocessor, one
adds directly partial derivatives of polynomials which are zero at the root.  
The postprocessor extracts a square subsystem of the completely deflated system 
for which the Jacobian has full rank at the root.

Consider the Jacobian matrix $J_{\fb} (\bx) = \left[\partial_{j} f_{i} (\bx)\right]$ of the
initial system $\fb$.
By reordering properly the rows and columns (i.e., polynomials and variables),
it can be put in the form 
\begin{equation} 
J_{\fb} (\bx) 
:= \left [
\begin{array}{cc}
A (\bx) & B (\bx) \\
C (\bx) & D (\bx)
\end{array}
\right] 
\end{equation}
where $A(\bx)$ is an $r\times r$ matrix with
$r = \rank J_{\fb}(\bxi) = \rank A(\bxi)$.

Suppose that $B(\bx)$ is an $r\times c$ matrix.  The $c$ columns
\begin{eqnarray*}\label{parkernel}
\det (A (\bx)) \left[
\begin{array}{c}
-A^{-1} (\bx) B (\bx) \\
\mathrm{Id}
\end{array}
\right]
\end{eqnarray*}
(for $r= 0$ this is the identity matrix) yield the $c$ elements 
$$ 
\Lambda_{1}^{\bx}=\sum_{i=1}^{n} \lambda_{1,j} (\bx) \partial_{j},~\ldots,~
\Lambda_{c}^{\bx}=\sum_{i=1}^{n}\lambda_{c,j} (\bx) \partial_{j}.
$$
Their coefficients $\lambda_{i,j}(\bx)\in \KK[\bx]$ are polynomial in the
variables $\bx$. 
Evaluated at $\bx=\bxi$, they generate the kernel of
$J_{\fb} (\bxi)$ and form a basis of $\DDD_{1}$. 
\begin{definition}
The family $D^{\bx}_{1}=\{\Lambda_{1}^{\bx}, \ldots,
\Lambda_{c}^{\bx}\}$ is the {\em formal} inverse system of
order $1$ at $\bxi$. 
For $\bm i=\{i_{1},\ldots, i_{k}\}\subset$ $\{1, \ldots, c\}$ 
with $|\bm i|\neq 0$, the $\bm i$-{\em deflated system} of order~$1$~of~$\fb$~is
$$ 
\{\fb, {\Lambda}_{i_{1}}^{\bx} (\fb), \ldots, {\Lambda}_{i_{k}}^{\bx} (\fb)\}.
$$
\end{definition}
By construction, for $i=1,\ldots,c$, 
$$
{\Lambda}_{i}^{\bx} (\fb) 
=  \sum_{j=1}^{n} \partial_{j} (\fb) {\lambda}_{i,j}(\bx) 
= \det (A (\bx)) J_{\fb} (\bx) [{\lambda}_{i,j}(\bx)]
$$ 
has $n-c$ zero entries. Thus,
the number of non-trivial new equations added in the
$\bm i$-deflated system
is \mbox{$|\bm i|\cdot(N-n+c)$}.
The construction depends on the choice of the invertible block $A(\bxi)$ in
$J_{\fb} (\bxi)$. 
By a linear invertible transformation of the initial system and by
computing a $\bm i$-deflated system, one obtains 
a deflated system constructed from any $|\bm i|$ linearly 
independent elements of the~kernel~of~$J_{\fb} (\bxi)$.

\begin{example}\label{Ex:Illustrative} Consider the multiplicity $2$ root $\bxi = (0,0)$ 
for the system $f_1(\bx) = x_1 + x_2^2$ and $f_2(\bx) = x_1^2 + x_2^2$.
Then,
{\small $$
J_{\fb}(\bx) =
\left[
\begin{array}{cc}
A (\bx) & B (\bx) \\
C (\bx) & D (\bx)
\end{array}
\right] = \left[\begin{array}{cc} 1 & 2x_2 \\ 2x_1 & 2x_2 \end{array}\right].
$$}
The corresponding vector $[-2x_2 ~~ 1]^T$ yields the element 
$$\Lambda_1^{\bx} = -2x_2\partial_1 + \partial_2.$$
Since $\Lambda_1^{\bx}(f_1) = 0$, the $\{1\}$-deflated system of 
order $1$ of $\fb$ is
$$
\left\{x_1 + x_2^2, ~~x_1^2 + x_2^2, ~-4x_1x_2 + 2x_2\right\}
$$
which has a multiplicity $1$ root at $\bxi$.
\end{example}
 
We use the following to analyze this deflation procedure.

\begin{lemma}[Leibniz rule]
For $a,b\in \KK[\bm x]$, 
$$ 
\bpartial^{\alpha} (a\,b) =\sum_{\beta\in \N^{n}} \frac{1}{\beta!} \bpartial^{\beta} (a) \pp{\partial}^{\bm\beta}
(\bpartial^{\alpha}) (b).
$$
\end{lemma}

\begin{proposition}\label{deflation:1} Let $r$ be the rank of
  $J_{\fb}(\bxi)$. Assume that $r<n$. Let $\bm i\subset
  \{1,\ldots,n\}$ with $0<|\bm i|\leq n-r$ 
  and $\fb^{(1)}$~be the $\bm i$-{\em deflated system} of order $1$ of
$\fb$. Then, $\mult_{\bxi}
(\fb^{(1)})\geq 1$ and $\nil_{\bxi} (\fb^{(1)}) < \nil_{\bxi} (\fb)$, which also implies that $\mult_{\bxi}
(\fb^{(1)})< \mult_{\bxi}
(\fb)$.
\end{proposition}
\begin{proof}
By construction, for $i\in\bm i$, 
the polynomials ${\Lambda}_{i}^{\bx} (\fb)$
vanish at $\bxi$, so that $\mult_{\bxi} (\fb^{(1)})\ge 1$. 
By hypothesis, the Jacobian of $\fb$ is not 
injective yielding $\nil_{\bxi}(\fb)> 0$.
Let $\DDD^{(1)}$ be the inverse
system of $\fb^{(1)}$ at $\bxi$. 
Since $(\fb^{(1)})\supset (\fb)$, 
we have $\DDD^{(1)}\subset \DDD$.
In particular, for any non-zero element $\Lambda \in \DDD^{(1)}\subset
\KK[\bm\partial_{\bxi}]$ and $i\in \bm i$, 
$\Lambda (\fb)=0$ and $\Lambda ({\Lambda}^{\bx}_{i}
(\fb))=0$.  

Using Leibniz rule, for any $p\in \KK[\bx]$, we have 
{\scriptsize
\begin{eqnarray*}
\Lambda ({\Lambda}_{i}^{\bx} (p)) &=&
\Lambda \left(\sum_{j=1}^{n} {\lambda}_{i,j} (\bx)\partial_{j} (p)\right)\\
&=&
\sum_{\beta\in \N^{n}} \sum_{j=1}^{n} \frac{1}{\beta!} \bpartial_{\bxi}^{\bm\beta} ( {\lambda}_{i,j} (\bx)) 
\pp{\partial_{\bxi}}^{\bm\beta}
(\Lambda) \partial_{j,\bxi} (p)\\ 
&=&
\sum_{\beta\in \N^{n}} \sum_{j=1}^{n} \frac{1}{\beta!} \bpartial_{\bxi}^{\bm\beta} ( {\lambda}_{i,j} (\bx)) \bpartial_{j,\xi} \pp{\partial_{\xi}}^{\bm\beta}
(\Lambda) (p)\\
&=&
\sum_{\beta\in \N^{n}} \Delta_{i,\beta} \pp{\partial_{\xi}}^{\bm\beta} (\Lambda) (p)\\
\end{eqnarray*}
} where 
{\small
$$ 
\Delta_{i,\bm\beta}=\sum_{j=1}^{n}
{\lambda}_{i,j,\bm\beta} \partial_{j,\xi}\in\KK[\bm\partial_{\bxi}]
\hbox{~and~}
{\lambda}_{i,j,\bm\beta} = \frac{1}{\bm\beta!}
\partial_{\bxi}^{\bm\beta}({\lambda}_{i,j}(\bx))\in \KK.$$
}

The term $\Delta_{i,\bm 0}$ is
$\sum_{j=1}^{n}
\lambda_{i,j} (\bxi) \partial_{j,\xi}$ which has 
degree $1$ in~$\bpartial_{\bxi}$ 
since $[\lambda_{i,j} (\bxi)]$ is a non-zero element of
$\ker J_{\fb} (\bxi)$.
For simplicity, let $\phi_{i}(\Lambda):= \sum_{\bm\beta\in \N^{n}} \Delta_{i,\bm\beta}
\pp{\partial}^{\bm\beta} (\Lambda)$.

For any $\Lambda\in \C[\bpartial_{\bxi}]$, we have
{\small
\begin{eqnarray*}
\pp{\partial_{j,\xi}} (\phi_{i}(\Lambda)) 
&=&\sum_{\beta\in \N^{n}} \lambda_{i,j,\beta}   \pp{\partial}^{\bm\beta}(\Lambda)+
\Delta_{i,\beta}  \pp{\partial}^{\bm\beta} (\pp{\partial_{j,\xi}}
(\Lambda))\\
&=&\sum_{\beta\in \N^{n}} \lambda_{i,j,\beta}   \pp{\partial}^{\bm\beta}(\Lambda)+
\phi_{i} (\pp{\partial_{j,\xi}} (\Lambda)).
\end{eqnarray*} 
} Moreover, if $\Lambda \in \DDD^{(1)}$, then by definition
$\phi_{i}(\Lambda) (\fb)=0$. 
Since $\DDD$ and $\DDD^{(1)}$ are both stable by derivation,
it follows that $\forall \Lambda \in \DDD^{(1)}$,
$\pp{\partial_{j,\xi}} (\phi_{i} (\Lambda))\in \DDD^{(1)}+ \phi_{i}(\DDD^{(1)})$. 
Since \mbox{$\DDD^{(1)}\subset \DDD$}, we know $\DDD+\phi_{i} (\DDD^{(1)})$ is stable by
derivation.  For any element $\Lambda$ of $\DDD+\phi_{i} (\DDD^{(1)})$,
$\Lambda (\fb)=0$.  We deduce that $\DDD+\phi_{i} (\DDD^{(1)})=\DDD$.
Consequently, the order of the elements in $\phi_{i} (\DDD^{(1)})$ 
is at most $\nil_{\bxi} (\fb)$.
The statement follows since $\phi_i$ increases the order by $1$,
therefore $\nil_{\bxi}(\fb^{(1)})< \nil_{\bxi}(\fb)$.
\end{proof}

We consider now a sequence of deflations of the system~$\fb$.
Let $\fb^{(1)}$ be the ${\bm i}_{1}$-deflated system of $\fb$. We
construct inductively  
$\fb^{(k+1)}$ as the ${\bm i}_{k+1}$-deflated system of $\fb^{(k)}$ for some
choices of ${\bm i}_{j}\subset \{1,\ldots,n\}$. 
 
\begin{proposition}
There exists $k\leq \nil_{\bxi} (\fb)$ such that $\bxi$ is a
simple root of $\fb^{(k)}$.
\end{proposition}
\begin{proof}
By Proposition \ref{deflation:1}, $\mult_{\bxi} (\fb^{(k)}) \geq 1$ and $\nil_{\bxi} (\fb^{(k)})$ is
strictly decreasing with $k$ until it reaches the value $0$.
Therefore, there exists $k\leq \nil_{\bxi}(I)$ such that $\nil_{\bm\xi}
(\fb^{(k)})=0$ and $\mult_{\bxi} (\fb^{(k)})\geq~1$.
This implies that $\bxi$ is a simple root of $\fb^{(k)}$.
\end{proof}
 
To minimize the number of equations added at each deflation step, we
take $|\bm i|=1$. Then, the number of non-trivial 
new equations added at each step is at most $N-n+c$.

We described this approach using first order differentials 
arising from the Jacobian, but this 
can be easily extended to use higher order differentials.

\section{The multiplicity structure}\label{Sec:PointMult}
Before describing our results, we start this section by recalling the  definition of orthogonal primal-dual  pairs of bases for the space $\K[\bx ]/Q$ and its dual. The following is a definition/lemma: 

\begin{lemma}[Orthogonal primal-dual basis pair]\label{pdlemma}
Let $\f$, $\bxi$,  $Q$, $\DDD$, $\mult= \mult_\xi(\f)$ and $\nil=\nil_\xi(\f)$ be as above. 
Then there exists a primal-dual basis pair of the local ring  $\K[\bx ]/ \QQ$ with the following properties:
\begin{itemize}
\item The {\em primal basis} of the local ring  $\K[\bx ]/ \QQ$ has the form 
\begin{equation}\label{pbasis}
B:=\left\{( \bx-\xi)^{\balpha_1},  ( \bx-\xi)^{\balpha_2},\ldots, ( \bx-\xi)^{\balpha_{\mult}}\right\}.
\end{equation}  
We can assume that  $\alpha_1=0$ and that the monomials in $B$ are  {\em connected to 1}
(c.f. \cite{Mourrain99-nf}). Define the set of exponents in $B$ 
$$
E:=\{\alpha_1, \ldots, \alpha_\mult\}.
$$ 
\item There is a unique {\em dual basis} $\bLambda\subset \DDD$ orthogonal to $B$, i.e. the elements of  $\bLambda$ are given in the following form:
\begin{eqnarray}\label{Macbasis}
\Lambda_{0}&=& \bpartial^{\balpha_1}_\bxi=1_\bxi\nonumber \\
\Lambda_{1}&=&\frac{1}{\balpha_1 !}\bpartial_\bxi^{\balpha_1}  +\sum_{|\bbeta|\leq \nil \atop \bbeta\not\in E}\nu_{\balpha_1, \bbeta} \;\bpartial_\bxi^{\bbeta}\nonumber\\
&\vdots&\\
\Lambda_{\mult-1}&=&\frac{1}{\balpha_{\mult} !}\bpartial_\bxi^{\balpha_{\mult}}  +\sum_{|\bbeta|\leq \nil\atop \bbeta\not\in E}\nu_{\balpha_\mult, \bbeta}\; \bpartial_\bxi^{\bbeta},\nonumber
\end{eqnarray} 
\item We have 
$0=\ord(\Lambda_{0}) \leq \cdots \leq \ord(\Lambda_{\mult-1})$, and  for  all    $0\leq t\leq \nil$ we have 
$$
\DDD_t={\rm span}\left\{ \Lambda_{j}\;:\;  \ord(\Lambda_{{j}})\leq t \right\},
$$ where $\DDD_{t}$ denotes the elements of $\DDD$ of order $\leq t$, as above.
\end{itemize}
\end{lemma}
\begin{proof}
  Let $\succ$ be the graded reverse lexicographic ordering in $\K[\bpartial]$ such that $\partial_{1}\prec \partial_{2}\prec \cdots \prec \partial_{n}$.
  We consider the initial $\mathrm{In}(\DDD)=\{\mathrm{In}(\Lambda)\mid \Lambda \in \DDD\}$ of $\DDD$ for the monomial ordering $\succ$.
  It is a finite set of increasing monomials
$D:=\left\{\bpartial^{\balpha_0},  \bpartial^{\balpha_1},\ldots, \bpartial^{\balpha_{\mult-1}}\right\},$
which are the leading monomials of the elements of
$\bLambda=\{\Lambda_{0},  \Lambda_{1},\ldots$, $\Lambda_{{\mult-1}}\} \subset \DDD.$
As $1\in \DDD$ and is the lowest monomial $\succ$, we have $\Lambda_{0}=1$.
As $\succ$ is refining the total degree in $\K[\bpartial]$, we have $\ord({\Lambda}_{i})=|\balpha_{i}|$ and
$0=\ord({\Lambda}_{0}) \leq \cdots \leq \ord({\Lambda}_{\mult-1})$.
Moreover, every element in $\DDD_{t}$ reduces to $0$ by the elements in $\bLambda$.
As only the elements $\Lambda_{{i}}$ of order $\le t$ are involved in this reduction, we deduce that
$\DDD_{t}$ is spanned by the elements $\Lambda_{{i}}$ with $\ord({\Lambda}_{i})\leq t$.

Let $E=\{ \balpha_{0},\ldots, \balpha_{\mult-1}\}$.
The elements $\Lambda_{{i}}$ are of the form
$$
\Lambda_{{i}}=\frac{1}{\balpha_{i} !}\bpartial_\bxi^{\balpha_{i}}  +\sum_{|\bbeta|\prec |\balpha_{i}|}\nu_{\balpha_i, \bbeta}\; \bpartial_\bxi^{\bbeta}.
$$
By auto-reduction of the elements  $\Lambda_{{i}}$, we can even suppose that $\bbeta\not\in E$ in the summation above, so that they are of the form \eqref{Macbasis}.

Let ${B}(\xi)=\left\{( \bx-\xi)^{{\balpha}_0}, \ldots, ( \bx-\xi)^{{\balpha}_{\mult-1}}\right\}\subset \K[\bx]
$. As $(\Lambda_{i}((\bx-\xi)^{{\balpha}_j}))_{0\le i,j\leq \mult-1}$ is the identity matrix, we deduce that
$B$ is a basis of $\K[\bx ]/ \QQ$, which is dual to $\bLambda$.

As $\DDD$ is stable by derivation, the leading term of $\frac{d }{d\partial_{i}}(\Lambda_{j})$ is in $D$.
If $\frac{d }{d\partial_{i}}(\bpartial_\bxi^{\balpha_{j}})$ is not zero, then it is the leading term of
$\frac{d }{d\partial_{i}}(\Lambda_{j})$, since the monomial ordering is compatible with the
multiplication by a variable. This shows that $D$ is stable by division by the variable $\partial_{i}$
and that $B$ is connected to $1$. This ends the proof of the lemma.  
\end{proof}
 
Such a basis of $\DDD$ can be obtained from any other basis of $\DDD$ by transforming first the coefficient matrix of the given dual basis into row echelon form and then reducing the elements above the pivot coefficients.
The integration method described in \cite{mantzaflaris:inria-00556021} computes a primal-dual pair,
such that the coefficient matrix has a block row-echelon form, each block being associated to an order.
The computation of a basis as in Lemma \ref{pdlemma} can be then performed order by order.

\begin{example} Let 
$$f_1=x_1-x_2+x_1^2, f_2= x_1-x_2+x_1^2,
$$
which has a multiplicity $3$ root at $\bxi=(0,0)$.  The integration method described in \cite{mantzaflaris:inria-00556021} computes a primal-dual pair 
$$
\tilde{B}=\left\{1,x_1,x_2\right\}, \; \tilde{\bLambda}=\left\{1, \partial_1+\partial_2, \partial_2+\frac{1}{2}\partial_1^2+ \partial_1\partial_2+\frac{1}{2}\partial_1^2\right\}.
$$
This primal dual pair does not form an orthogonal pair, since $( \partial_1+\partial_2)(x_2)\neq 0$. However, using let say the degree lexicographic ordering such that $x_1>x_2$, we easily deduce the primal-dual pair of Lemma \ref{pdlemma}:
$$
{B}=\left\{1,x_1,x_1^2\right\}, \quad \bLambda=\tilde{\bLambda}.
$$
\end{example}

Throughout this section we assume that we are given a fixed primal basis $B$ for $\K[\bx ]/ \QQ$
such that a dual basis $\bLambda$ of $\DDD$ satisfying the properties of Lemma \ref{pdlemma} exists. Note that such a primal basis $B$  can be computed numerically  from an approximation of $\xi$ and using a modification of the integration method of  \cite{mantzaflaris:inria-00556021}.

Given the primal basis $B$, a dual basis can be computed by
Macaulay's dialytic method which can be used to deflate the root $\bxi$ as
in \cite{lvz08}. This method would introduce
\mbox{$n+(\mult-1)
\left({{n+\nil}\choose{n}}-\mult\right)$} new variables, which is not polynomial in $\nil$. Below, we give a construction of a 
polynomial system that only depends on at most 
$n+ n\mult(\mult-1)/2$ variables. These variables 
correspond to the entries of the {\em multiplication matrices} that we~define~next.
Let 
\begin{eqnarray*}
M_{i} : \K[\bx]/Q&\rightarrow &  \K[\bx]/Q\\
  p & \mapsto & (x_{i}-\xi_{i})\, p
\end{eqnarray*} 
be the multiplication operator by $x_{i}-\xi_{i}$ in 
$\K[\bx]/Q$. Its transpose operator is
\begin{eqnarray*}
M_{i}^{t} : \DDD&\rightarrow &  \DDD\\
  \Lambda & \mapsto & \Lambda \circ M_{i}= (x_{i}-\xi_{i})\cdot \Lambda = \frac{d}{d\partial_{i,\xi}} (\Lambda)=d_{\partial_{i, \xi}}(\Lambda)
\end{eqnarray*} 
where $\DDD= Q^{\perp}\subset \K[\bpartial]$. The matrix of
$M_{i}$ in the basis $B$ of  $\K[\bx]/Q$ is denoted $\mM_{i}$.

As $B$ is a basis of $\K[\bx]/Q$, we can identify the elements of
$\K[\bx]/Q$ with the elements of the vector space $\sp_\K( B)$. 
We define the normal form $N(p)$ of a polynomial $p$ in $\K[\bx]$ as the
unique element $b$ of ${\rm span}_\K(B)$ such
that $p-b\in Q$. Hereafter, we are going to identify the elements of
$\K[\bx]/Q$ with their normal form in $\sp_\K (B )$.

For any polynomial $p (x_{1}, \ldots, x_{n}) \in \K[\bx]$, let $p
(\bM)$ be the operator of $\K[\bx]/Q$ obtained by replacing $x_{i}-\xi_{i}$ by $M_{i}$.

\begin{lemma}
For any $p\in \K[\bx]$, the normal form of $p$ is $N (p)= p (\bM) (1)$ and we
have
$$ 
p (\bM) (1) = \Lambda_{0}(p)\, 1 + \Lambda_{1}(p) \, ( \bx-\bxi)^{\balpha_1}+\cdots + \Lambda_{{\mult-1}}(p)\, ( \bx-\bxi)^{\balpha_{\mult-1}}.
$$
\end{lemma}

This shows that the coefficient vector $[p]$ of $N (p)$ in the basis $B$ of
 is $[p]= (\Lambda_{{i}}(p))_{0\le i \le \mult-1}$.
 
The  following lemma is also well known, but we include it here with proof:
\begin{lemma}\label{multlemma} Let $B$ as in (\ref{pbasis}) and denote the exponents in $B$ by 
$
E:=\{\alpha_1, \ldots, \alpha_\mult\}
$ as above.
Let  
$$E^+:=\bigcup_{i=1}^n (E+ \e_i )$$ 
with $E+\e_i=\{(\gamma_1, \ldots , \gamma_i+1, \ldots, \gamma_n):\gamma\in E\}$ and 
we denote $\partial(E)= E^{+}\setminus E$. 
The values of the coefficients $\nu_{\alpha, \beta}$
for $(\alpha,\beta)\in E\times \partial(E)$ appearing in the dual basis (\ref{Macbasis})
uniquely determine the system of pairwise commuting multiplication  matrices $\mM_{i}$, namely,  
for $i=1, \ldots, n$ 
\begin{eqnarray}\label{multmat}
\mM_{i}^{t}=
\begin{array}{|ccccc|}
\cline{1-5}
 0&\nu_{\balpha_1, \e_i}&\nu_{\balpha_2, \e_i}&\cdots &\nu_{\balpha_{d-1}, \e_i}  \\ 
 0&0&\nu_{\balpha_2,\balpha_1+\e_i}&\cdots &\nu_{\balpha_{d-1},\balpha_1+\e_i}\\  
 \vdots &\vdots &&&\vdots\\
0&0&0&\cdots &\nu_{\balpha_{d-1},\balpha_{d-2}+\e_i}\\ 
0&0&0&\cdots &0\\
\cline{1-5} 
\end{array}
\end{eqnarray} 
Moreover, 
$$
\nu_{\alpha_i, \alpha_k+\e_j}=\begin{cases} 1 & \text{ if } \alpha_i= \alpha_k+\e_j\\
0 & \text{ if } \alpha_k+\e_j \in E, \; \alpha_i \neq \alpha_k+\e_j .
\end{cases}
$$
\end{lemma}
\begin{proof}
As $M_{i}^{t}$ acts as a derivation on $\DDD$ and $\DDD$ is closed under derivation,  so the third property in Lemma \ref{pdlemma} implies that the matrix
of $M_{i}^{t}$ in the basis of $\bLambda$ of $\DDD$ has an upper triangular form with
zero (blocks) on the diagonal.

For an element $\Lambda_{{j}}$ of order $k$, its image by
$M_{i}^{t}$ is 
{\small \begin{eqnarray*}
&&M_{i}^{t} (\Lambda_{{j}}) =
(x_{i} - \xi_{i}) \cdot \Lambda_{{j}}\\
&&=\sum_{|\balpha_{l}|<k} \Lambda_{{j}} ((x_{i}-\xi_{i})
(\bx-\bxi)^{\balpha_{l}}) \Lambda_{{l}}\\
 &&=  \sum_{|\balpha_{l}|<k} \Lambda_{{j}}
((\bx-\bxi)^{\balpha_{l}+\e_{i}}) \, \Lambda_{{l}}
 =  \sum_{|\balpha_{l}|<k} 
\nu_{\balpha_{j}, \balpha_l+ \e_i} \Lambda_{{l}}.
\end{eqnarray*} }
This shows that the entries of $\mM_{i}$ are the coefficients of the
dual basis elements corresponding to exponents in $E\times \partial(E)$. The second claim is clear from the definition of $\mM_{i}$.
\end{proof}

The previous lemma shows that the dual basis uniquely defines the system of multiplication matrices for $  i=1, \ldots, n$
{\small 
\begin{eqnarray*}
\mM_{i}^{t}&=&\begin{array}{|ccc|}
\cline{1-3}
 \Lambda_{{0}}(x_i-\xi_i)&\cdots &\Lambda_{{\mult-1}}(x_i-\xi_i)  \\ 
 \Lambda_{{0}}\left(( \bx-\bxi)^{\balpha_1+\e_i}\right)&\cdots &\Lambda_{{\mult-1}}\left(( \bx-\bxi)^{\balpha_1+\e_i}\right) \\  
 \vdots &&\vdots\\
 \Lambda_{{0}}\left(( \bx-\bxi)^{\balpha_d+\e_i}\right)&\cdots &\Lambda_{{\mult-1}}\left(( \bx-\bxi)^{\balpha_\mult+\e_i}\right) \\  
\cline{1-3} 
\end{array}\nonumber\\
&=& 
\begin{array}{|ccccc|}
\cline{1-5}
  0&\nu_{\balpha_1, \e_i}&\nu_{\balpha_2, \e_i}&\cdots &\nu_{\balpha_{\mult-1}, \e_i}  \\ 
 0&0&\nu_{\balpha_2,\balpha_1+\e_i}&\cdots &\nu_{\balpha_{\mult-1},\balpha_1+\e_i}\\  
 \vdots &\vdots &&&\vdots\\
0&0&0&\cdots &\nu_{\balpha_{\mult-1},\balpha_{\mult-2}+\e_i}\\ 
0&0&0&\cdots &0\\
\cline{1-5} 
\end{array}
\end{eqnarray*}}
Note that these matrices are nilpotent by their upper triangular
structure, and all $0$ eigenvalues. 
As $\nil$ is the maximal order of the elements of $\DDD$, we have
$\mM^{\bgamma}=0$ if $|\bgamma|> \nil$.\\


Conversely, the system of multiplication matrices $\mM_1, \ldots, \mM_n$ uniquely defines the dual basis as follows. Consider $\nu_{\balpha_i,\bgamma}$  for some  $(\balpha_i, \bgamma)$ such that $|\bgamma|\leq \nil$ but $\bgamma \not\in E^+$.  We can uniquely determine $\nu_{\balpha_i,\bgamma}$ from the values of $\{\nu_{\balpha_j, \bbeta}\; : \;(\balpha_j, \bbeta)\in E\times \partial(E)\}$  from the following identities:
\begin{equation}\label{restnu}
\nu_{\balpha_i,\bgamma}= \Lambda_{i}((\bx -\bxi)^{\bgamma})
=[\mM_{(\bx -\xi)^{\bgamma}}]_{1, i}= [\mM^{\bgamma}]_{1,i}.
\end{equation}

The next definition defines the {\em parametric multiplication matrices} that we use  in our constriction.

\begin{definition}[Parametric multiplication matrices] Let $B$ as in (\ref{pbasis}),  and  $E$, $\partial(E)$ as in Lemma \ref{multlemma}.  We define array  ${\bmu}$ of length $n\mult(\mult-1)/2$ consisting of $0$'s, $1$'s and the variables 
  $\mu_{\alpha_i, \beta}$
  as follows: for all $\alpha_i, \alpha_k\in E$ and $j\in \{1, \ldots, n\}$ the corresponding entry is 
\begin{eqnarray}\label{defmuE}
{\bmu}_{\alpha_i, \alpha_k+\e_j}=\begin{cases} 1 & \text{ if } \alpha_i= \alpha_k+\e_j\\
0 & \text{ if } \alpha_k+\e_j \in E, \; \alpha_i \neq \alpha_k+\e_j \\
\mu_{\alpha_i, \alpha_k+\e_j} & \text{ if  } \alpha_k+\e_j \not\in E.
\end{cases}
\end{eqnarray}
The {\em parametric multiplication matrices} are defined  
for $i=1, \ldots, n$ by
\begin{equation}\label{Mmu}
\mM_{i}^{t}({\bmu}):= \begin{array}{|ccccc|}
\cline{1-5}
  0&\bmu_{\balpha_1, \e_i}&\bmu_{\balpha_2, \e_i}&\cdots &\bmu_{\balpha_{\mult-1}, \e_i}  \\ 
 0&0&\bmu_{\balpha_2,\balpha_1+\e_i}&\cdots &\bmu_{\balpha_{d-1},\balpha_1+\e_i}\\  
 \vdots &\vdots &&&\vdots\\
0&0&0&\cdots &\bmu_{\balpha_{d-1},\balpha_{\mult-2}+\e_i}\\ 
0&0&0&\cdots &0\\
\cline{1-5} 
\end{array} ,
\end{equation}
We denote by 
$$
\mM(\bmu)^\bgamma:=\mM_1(\bmu)^{\gamma_1}\cdots\mM_n(\bmu)^{\gamma_n},
$$
and note  that for general parameters values $\bmu$, the matrices  $\mM_i(\mu)$ do not commute, so we fix their order by their indices in the above definition of  $\mM(\mu)^\bgamma$. 
\end{definition}

\begin{remark} \label{reduce} Note that we can reduce the number of free parameters in the parametric multiplication matrices by exploiting the commutation rules of the multiplication matrices corresponding to a given  primal basis $B$. For example, consider the breadth one case, where we can assume that $E=\{{\bf 0}, \e_1, 2\e_1, \ldots, (\delta-1)\e_1\}$. In this case the only free parameters appear in the first columns of $\mM_2(\mu), \ldots, \mM_n(\mu)$,  the other columns are shifts of these.   Thus, it is enough to introduce $ (n-1)(\delta-1)$ free parameters,  similarly as in \cite{LiZhi2013}. In Section \ref{Sec:Examples} we present a modification of \cite[Example 3.1]{LiZhi2013} which has breadth two, but also uses at 
most $ (n-1)(\delta-1)$ free parameters.
\end{remark}

\begin{definition}[Parametric normal form] \label{parnorm} Let $\KK\subset \K$ be a field. We define
\begin{eqnarray*}
\Nc_{\bz,\bmu} : \KK[\bx] & \rightarrow & \KK[\bz,\bmu]^{\mult}\\
p& \mapsto& \Nc_{\bz,\bmu} (p) :=  \sum_{\bgamma\in \N^n}
\frac{1}{\bgamma!} \bpartial_{\bz}^{\bgamma}(p) \, \mM({\bmu})^{\bgamma}[1].
\end{eqnarray*}
where $[1]=[1,0,\ldots,0]$ is the coefficient vector of $1$ in the basis $B$. This sum is finite since for $|\bgamma|\geq \mult$, 
$\mM({\bmu})^{\bgamma}=0$, so the entries of $\Nc_{\bz,\bmu} (p)$ are polynomials in ${\bmu}$ and $\bz$.

\end{definition}

Notice that this notation is not ambiguous, assuming that the matrices $\mM_{i}(\mu)$
($i=1,\ldots,n$)  are commuting.
The specialization at $(\bx,\bmu)= (\bxi,\bnu)$ is the vector
$$ 
\Nc_{\bxi,\bnu} (p) =[\Lambda_{{0}} (p),
\ldots, \Lambda_{{\mult-1}} (p)]^{t}\in \K^{\mult}.
$$

\subsection{The multiplicity structure equations of a singular point}

\noindent We can now characterize the multiplicity structure by
polynomial equations.
\begin{theorem} \label{theorem1}
Let $\KK\subset \C$ be any field, $\f\in\KK[\bx]^N$ and let $\bxi\in
\C^n$ be an isolated solution of $\f$.  Let $\mM_i(\bmu)$ for $i=1, \ldots n$ be the parametric multiplication matrices as in~(\ref{Mmu}) and~$\Nc_{\bxi,\bmu}$ be the parametric normal form as in Defn.~\ref{parnorm} at $\bz=\bxi$.  
Then the ideal $J_{\bxi}$ of $\C[\bmu]$ generated by the polynomial system 
{\small\begin{eqnarray}\label{matrixeq}
\begin{cases} 
\Nc_{\bxi,\bmu} (f_{k})\;\; \text{ for } k=1, \ldots, N, \\
\mM_{i}({\bmu})\cdot \mM_{j}({\bmu})-\mM_{i}({\bmu})\cdot \mM_{i}({\bmu})\;\;\text{ for } i, j=1, \ldots, n
\end{cases}
\end{eqnarray}}
is the maximal ideal
$$
\m_{\nu}= (\bmu_{\balpha, \bbeta}- \nu_{\balpha, \bbeta}, (\balpha,\bbeta)\in E\times \partial(E))
$$ 
where $\nu_{\balpha, \bbeta}$ are the coefficients of the dual basis defined in~(\ref{Macbasis}). 
\end{theorem}
\begin{proof}
 As before, the system \eqref{matrixeq} has a solution
$\bmu_{\balpha,\bbeta}=\nu_{\balpha,\bbeta}$ for $(\balpha,\bbeta)\in
E\times \partial(E)$. Thus $J_{\bxi}\subset \m_{\nu}$.

Conversely, let $C=\K[\mu]/J_{\bxi}$ and consider the map 
$$ 
 \Phi: C[\bx] \rightarrow C^{\mult}, \;\;
p  \mapsto  \Nc_{\bxi ,\bmu}(p).
$$
Let $K$ be its kernel.
Since the matrices $\mM_{i} (\bmu)$ are commuting
modulo $J_{\bxi}$, we can see that  $K$ is an ideal. 
As $f_{k}\in K$, we have $\I:=(f_{k}) \subset K$. 

Next we show that $Q\subset K$.
By construction, for any $\alpha\in \N^{n}$ we have modulo $J_\xi$
\begin{equation*}\label{eq:rel1}
\Nc_{\bxi ,\bmu}((\bx-\bxi)^{\alpha})= \sum_{\bgamma\in \N^n}
\frac{1}{\bgamma!} \bpartial_{\bxi}^{\bgamma}((\bx-\bxi)^{\alpha}) \, \mM({\bmu})^{\bgamma}[1] =  \mM({\bmu})^{\alpha}[1].
\end{equation*}
Using the previous relation, we check that $\forall p,q \in C[\bx]$, 
\begin{equation}\label{eq:prod}
\Phi(p q) = p(\bxi+\mM (\bmu)) \Phi(q)
\end{equation}
where $ p(\bxi+\mM (\bmu))$ is obtained by replacing $x_{i}-\xi_{i}$ by $\mM_{i} (\bmu)$.
Let $q\in Q$. As $Q$ is the $\m_{\bxi}$-primary component of $\I$,
there exists $p\in \C[\bx]$ such that $p (\bxi)\neq 0$ and $p\, q\in
\I$. By~\eqref{eq:prod},~we~have 
$$ 
\Phi (p\, q)= p(\bxi+\mM (\bmu)) \Phi (q) = 0.
$$
Since $p (\bxi)\neq 0$ and $ p(\bxi+\mM (\bmu))= p (\bxi) Id + N$ with $N$
lower triangular and nilpotent, $p(\bxi+\mM (\bmu))$ is invertible.
We deduce that $\Phi(q)= p (\bxi+\mM (\bmu))^{-1}\Phi (pq) =0$ and $q \in K$.

Let us show now that $\Phi$ is surjective and more precisely, that
$\Phi ((\bx-\bxi)^{\alpha_{k}})=\e_{k}$  (abusing the notation as here
$\e_k$ has length $\mult$ not $n$). Since $B$ is connected to $1$,
either $\alpha_k=0$ or there exists $\alpha_j\in E$ such that
$\alpha_k=\alpha_j+\e_i$ for some $i\in \{1, \ldots, n\}$. Thus the
$j^{\rm th}$ column of $\mM_i({\bmu})$ is $\e_k$ by (\ref{defmuE}).  As
$\{\mM_{i}({\bmu}): i=1, \ldots, n\}$ are  pairwise commuting, we have
$\mM({\bmu})^{\alpha_k}=\mM_j(\bmu) \mM({\bmu})^{\alpha_j}$, and if we
assume by induction on $|\alpha_{j}|$ that the first column of $\mM({\bmu})^{\alpha_j}$
is $\e_j$, we obtain \mbox{$\mM({\bmu})^{\alpha_k}[1]=\e_k$}. 
Thus, for $k = 1,\dots,\mult$,  
$\Phi ((\bx-\bxi)^{\alpha_{k}})=\e_{k}$.

We can now prove that $\m_{\nu}\subset J_{\bxi}$. As $M_{i}(\nu)$ is the multiplication by $(x_{i}-\xi_{i})$
in $\C[\bx]/Q$, for any $b\in B$ and $i=1,\ldots,n$, we have
$(x_{i}-\xi_{i})\, b = M_{i}(\nu) (b) + q$ with $q\in Q\subset K$. We deduce that for $k=0,\ldots,\mult-1$,
$$\hbox{\scriptsize 
$\Phi ((x_{i}-\xi_{i})\, (\bx -\bxi)^{\alpha_{k}}) =  \mM_{i}(\bmu) \Phi
((\bx -\bxi)^{\alpha_{k}}) = \mM_{i}(\bmu) (\e_{k}) =
\mM_{i}(\nu) (\e_{k})$.}$$

This shows that $\bmu_{\alpha,\beta}-\nu_{\alpha,\beta}\in J_{\bxi}$
for $(\alpha,\beta) \in E\times\partial(E)$ and that $\m_{\nu}=J_{\bxi}$.
\end{proof}

In the proof of the next theorem  we  need to consider cases when the multiplication matrices do not commute. We introduce the following definition:

\begin{definition}\label{commutator}
Let $\KK\subset \C$ be any  field. Let $\Cc$ be the ideal of $\KK[\bz,\mu]$ generated by entries of the commutation
relations:
$\mM_{i}({\bmu})\cdot \mM_{j}({\bmu})-\mM_{j}({\bmu})\cdot
\mM_{i}({\bmu})=0$, $i,j=1,\ldots,n$. We call $\Cc$ the {\em commutator ideal}.
\end{definition}

\begin{lemma} For any field $\KK\subset \C$, $p\in \KK[\bx]$, and $i=1,\ldots,n$, we have
\begin{equation}
\Nc_{\bz,\bmu} ( x_{i} p) = x_{i} \Nc_{\bz, \bmu} (p) + \mM_{i} (\bmu)\, \Nc_{\bz, \bmu} (p) + O_{i, \bmu}(p). \label{eq:nf}
\end{equation}
where $O_{i, \mu}: \KK[\bx]\rightarrow \KK[\bz, \mu]^{\mult}$ is linear with image in the commutator ideal $\Cc$.
\end{lemma} 
\begin{proof}
$\Nc_{\bz, \bmu} (x_{i} p) 
=  
\sum_{\bgamma}
\frac{1}{\bgamma!} \, \partial_{\bz}^{\bgamma} (x_{i} p) \,
  \mM(\bmu)^{\bgamma}[1]
$
  \begin{eqnarray*}
& = &
x_{i} \sum_{\bgamma} \frac{1}{\bgamma!} \partial_{\bz}^{\bgamma} (p) \, \mM({\bmu})^{\bgamma}[1] +
\sum_{\bgamma} 
\frac{1}{\bgamma!} \gamma_{i}\, \partial_{\bz}^{\bgamma-e_{i}} (p) \,
\mM({\mu})^{\bgamma}[1]\\ 
& =&
x_{i} \sum_{\bgamma} \frac{1}{\bgamma!} \partial_{\bz}^{\bgamma} (p)
\, \mM({\bmu})^{\bgamma}[1] 
+ \sum_{\bgamma} \frac{1}{\bgamma!} \partial_{\bz}^{\bgamma} (p) \,
\mM({\bmu})^{\bgamma+e_{i}}[1]\\
& = &
x_{i} \, \Nc_{\bz, \bmu} (p) 
+ \mM_{i}(\bmu) \left ( \sum_{\bgamma} \frac{1}{\bgamma!} \partial_{\bz}^{\bgamma} (p) \,
\mM({\bmu})^{\bgamma}[1] \right) \\
&&~~~~~~+ \sum_{\bgamma} \frac{1}{\bgamma!} \, \partial_{\bz}^{\bgamma} (p) \,
O_{i, \bgamma}({\bmu})[1]  \\
\end{eqnarray*}
where $O_{i, \bgamma}= \mM_{i} (\bmu) \mM (\bmu)^{\bgamma}- \mM (\bmu)
^{\bgamma+e_{i}}$ is a $\mult \times \mult$ matrix with coefficients in $\Cc$.
Therefore, $O_{i, \mu}:p\mapsto\sum_{\bgamma} \frac{1}{\bgamma!} \partial_{\bz}^{\bgamma} (p) \,
O_{i,\bgamma}({\bmu})[1]$ is a linear functional of $p$ with coefficients
in $\Cc$.
\end{proof}

The next theorem proves that the system defined as in~(\ref{matrixeq}) for general $\bz$ has $(\bxi, {\nu})$ as a simple root.

\begin{theorem}
Let $\f\in \KK[\bx]^N$ and $\bxi\in \C^n$ be as above. Let $\mM_i(\bmu)$ for $i=1, \ldots n$ be the parametric multiplication matrices defined in (\ref{Mmu}) and $\Nc_{\bx,\mu}$ be the parametric normal form as in Defn.~\ref{parnorm}.
Then $(\bz, {\mu})=(\bxi, {\nu})$ is an isolated  root with multiplicity one of the polynomial system in  $\KK[\bz, {\mu}]$:
{\small \begin{eqnarray}\label{overdet}
\begin{cases}  
\Nc_{\bz,\bmu} (f_{k}) = 0 \;\text{ for } k=1, \ldots, N,\\
\mM_{i}({\bmu})\cdot \mM_{j}({\bmu})-\mM_{j}({\bmu})\cdot \mM_{i}({\bmu})=0\;
\text{ for } i, j=1, \ldots, n.  \end{cases}
\end{eqnarray}}
\end{theorem}
\begin{proof}
  For simplicity, let us denote the (non-zero) polynomials appearing in (\ref{overdet}) by 
$$P_1, \ldots, P_M\in  \KK[\bz, {\bmu}],$$ 
where $M\leq N\mult+n(n-1)(\mult-1)(\mult-2)/4$. To prove the theorem, it is sufficient to prove that the columns of the Jacobian matrix of the system $[P_1, \ldots, P_M]$ at $(\bz, {\bmu})=(\bxi, {\nu})$ are linearly independent. The columns of this Jacobian matrix correspond to the elements in $\K[\bz, {\bmu}]^*$ 
$$\partial_{1, \xi}, \ldots, \partial_{n, \xi}, \text{ and } \partial_{\mu_{\balpha, \bbeta}}\;\text{ for } \; (\balpha, \bbeta) \in E\times \partial(E),
$$ 
where $\partial_{i, \xi}$ defined in (\ref{partial}) for  $\bz$ replacing $\bx$, and $\partial_{\bmu_{\balpha, \bbeta}}$ is defined by
$$
\partial_{{\bmu_{\balpha, \bbeta}}}(q) = \frac{d q}{ d \bmu_{\balpha, \bbeta}}\left|_{(\bz, {\mu})=(\bxi, {\nu})} \right. \quad \text{ for } q\in \K[\bz, {\mu}]. 
$$
Suppose there exist $a_1, \ldots, a_n,$ and $a_{\balpha, \bbeta}\in \C$ for  $(\balpha,\bbeta) \in  E\times \partial(E)$ not all zero 
such that 
$$
\Delta:= a_1\partial_{1, \xi}+ \cdots + a_n\partial_{n, \xi}+\sum_{\balpha, \bbeta} a_{\balpha, \bbeta} \partial_{\mu_{\balpha, \bbeta}}\in \K[\bz, {\bmu}]^*
$$
vanishes on all polynomials $P_1, \ldots, P_M$ in (\ref{overdet}).  In
particular, for an element $P_{i} (\bmu)$ corresponding to the commutation
relations and any polynomial $Q  \in \K[\bx, \mu]$, using the product rule for the linear differential operator $\Delta$ we get  
$$
\Delta (P_{i} Q)= \Delta (P_{i}) Q (\bxi,\bnu) + P_{i} (\bnu) \Delta (Q) = 0
$$
since $ \Delta (P_{i}) =0$ and $P_{i} (\bnu)=0$. By the linearity of $\Delta$, for any
polynomial $C$ in the commutator ideal $ \Cc$, we have $\Delta (C)=0$.

Furthermore, since $\Delta(\Nc_{\bz, \bmu}(f_k))=0$ and $$\Nc_{\bxi, \bnu} (f_k) = [\Lambda_{0}(f_k), \ldots, \Lambda_{{\mult-1}}(f_k)]^t,$$ we get that 
{\small \begin{equation}\label{dualeq}
 (a_1\partial_{1, \xi}+ \cdots+a_n\partial_{n, \xi})\cdot \Lambda_{{\mult-1}}(f_k)+ \sum_{|\bgamma|\leq |\balpha_{\mult-1}|}p_{\bgamma}({\nu}) \; \bpartial^{\bgamma}_{\xi}(f_k)=0
\end{equation}}
where $p_{\bgamma}\in \K[{\mu}]$ are  some polynomials in the
variables $\mu$ that do not depend on $f_k$. 
If $a_{1}, \ldots, a_{n}$ are not all zero, we have an element $\tilde{\Lambda}$ of $\K[\bpartial_\bxi]$ of order strictly greater than
${\rm ord} (\Lambda_{{\mult-1}})=\nil$ that vanishes on $f_1, \ldots, f_N$. 

Let us prove that this higher order differential also vanishes on all multiples of  $f_k$ for
$k=1, \ldots, N$.
Let $p\in \K[\bx]$ such that $\Nc_{\bxi,\bnu} (p)=0$, $\Delta
(\Nc_{\bz,\bmu} (p))=0$. By~\eqref{eq:nf}, we have
\begin{eqnarray*}
\lefteqn{\Nc_{\bxi,\bnu} ((x_{i}-\xi_{i}) p)}\\ &= &
(x_{i}-\xi_{i}) \Nc_{\bxi,\bnu} (p) + 
\mM_{i} (\nu) \Nc_{\bxi,\bnu} (p) + O_{i,\nu} (p) =  0
\end{eqnarray*}
and ${\Delta (\Nc_{\bz, \bmu} ((x_{i}-\xi_{i}) p))}$
\begin{eqnarray*}
&= &
\Delta ((x_{i}-\xi_{i}) \Nc_{\bz, \bmu} (p)) + 
\Delta (\mM_{i} (\mu) \Nc_{\bz, \bmu} (p)) + \Delta( O_{\mu} (p))\\
 & = &
\Delta (x_{i}-\xi_{i}) \Nc_{\bxi, \bnu} (p) + (\xi_{i}-\xi_{i})\Delta( \Nc_{\bz, \bmu} (p)) \\
&& ~~~~ + \Delta (\mM_{i} (\mu))  \Nc_{\bxi, \bmu} (p) + 
\mM_{i} (\nu)  \Delta (\Nc_{\bz, \bmu} (p)) \\
&& ~~~~ +  \Delta (O_{i,\bmu} (p))\\
& = & 0.
\end{eqnarray*}
As $\Nc_{\bxi,\bnu} (f_{k})=0$, $\Delta (\Nc_{\bz,\bmu} (f_{k}))=0$,
$i=1,\ldots, N$, we
deduce by induction on the degree of the multipliers and by linearity that for any
element $f$ in the ideal $I$ generated by $f_{1}, \ldots, f_{N}$, we
have 
$$ 
\Nc_{\bxi,\bnu} (f)=0 \hbox{~~~and~~~} \Delta (\Nc_{\bz,\bmu} (f))=0,
$$
which yields $\tilde{\Lambda} \in I^{\bot}$. Thus we have
$\tilde{\Lambda} \in I^{\bot}\cap \K[\bpartial_\bxi]= Q^{\bot}$ (by Lemma
\ref{lem:primcomp}).
As there is no element of degree strictly bigger than $\nil$ in
$Q^{\bot}$, this implies that 
$$a_1=\cdots=a_n=0.$$
Then, by specialization at $\bx=\bxi$, $\Delta$ yields an element of the kernel
of the Jacobian matrix of the system \eqref{matrixeq}.
By Theorem \ref{theorem1}, this Jacobian has a zero-kernel, since it defines
the simple point $\nu$. We deduce that $\Delta=0$ and $(\bxi,\bnu)$ is
an isolated and simple root of the system  \eqref{overdet}.
\end{proof}

The following corollary applies the polynomial system defined in (\ref{overdet}) to refine the precision of an approximate multiple root together with the coefficients of its Macaulay dual basis. The advantage of using this, as opposed to using the Macaulay multiplicity matrix, is that the number of variables is much smaller, as was noted above.  

\begin{corollary} Let $\f\in \KK[\bx]^N$ and $\bxi\in \C^n$ be as above, and let $\Lambda_{0}({\nu}), \ldots, \Lambda_{{\mult-1}}({\nu})$ be its dual basis as in~(\ref{Macbasis}).  Let $E\subset \N^n$ be as above. Assume that we are given  approximates for the singular roots and its inverse system as in (\ref{Macbasis})
$$
\tilde{\bxi} \cong \bxi \; \text{ and } \; \tilde{\nu}_{\alpha_i, \bbeta}\cong \nu_{\alpha_i, \bbeta} \;\;\forall \balpha_i \in E,\; \beta\not\in E, \;|\bbeta|\leq \nil.
$$
Consider the overdetermined system in $\KK[\bz, \mu]$ from (\ref{overdet}).
Then a random square subsystem of (\ref{overdet}) will have 
a simple root at $\bz=\bxi$, $\mu=\nu$ with high probability. Thus,  we can apply Newton's method for this square subsystem to refine $\tilde{\bxi}$ and $\tilde{\nu}_{\alpha_i, \bbeta}$ for $(\balpha_i, \bbeta)\in E\times \partial(E)$. For   $\tilde{\nu}_{\alpha_i, \bgamma}$ with $\bgamma\not \in E^+$ we can use (\ref{restnu}) for the update.  
\end{corollary}

\begin{example}\label{Ex:Illustrative2}
Reconsider the setup from Ex.~\ref{Ex:Illustrative} with primal
basis $\{1,x_2\}$ and $E = \{(0,0),(0,1)\}$.  We obtain
$$\mM_1(\mu) = \left[\begin{array}{cc} 0 & 0 \\ \mu & 0 \end{array}\right]~~\hbox{and}~~
\mM_2(\mu) = \left[\begin{array}{cc} 0 & 0 \\ 1 & 0 \end{array}\right].$$
The resulting deflated system in (\ref{overdet}) is
{\small
$$F(z_1,z_2,\mu) = \left[\begin{array}{c} z_1 + z_2^2 \\
\mu + 2 z_2 \\ z_1^2 + z_2^2 \\ 2 \mu z_1 + 2 z_2 \end{array}\right]$$
}
which has a nonsingular root at $(z_1,z_2,\mu) = (0,0,0)$ corresponding
to the origin with multiplicity structure $\{1,\partial_{2}\}$.
\end{example}

\section{Examples}\label{Sec:Examples}
Computations for the following examples, as well as several other 
systems, along with \textsc{Matlab} code can be found at
\url{www.nd.edu/~jhauenst/deflation/}.

\subsection{A family of examples}
\noindent We first consider a modification of \cite[Example 3.1]{LiZhi2013}. For any $n\geq 2$, the following system has $n$ polynomials, each of degree at most $3$, in $n$ variables:
\begin{eqnarray*} 
x_1^3+x_1^2-x_2^2, \;x_2^3+x_2^2-x_3, \ldots, x_{n-1}^3+x_{n-1}^2-x_n, \;x_n^2.
\end{eqnarray*}
The origin is a multiplicity $\delta:=2^n$ root having breadth $2$ (i.e., the 
corank of Jacobian at the origin is $2$). 

We apply our parametric normal form method described in \S~\ref{Sec:PointMult}. Similarly as in Remark \ref{reduce}, we can reduce the number of free parameters to be at most $(n-1)(\delta-1)$ using the structure of the primal basis $B=\{x_1^ax_2^b:a<2^{n-1}, \;  b<2\}$.

The following table shows the multiplicity, number of
variables and polynomials in the deflated system, and the time (in
seconds) it took to compute this system (on a iMac, 3.4 GHz Intel Core i7 processor, 8GB 1600Mhz DDR3 memory).
Note that when comparing our method to an 
approach using the
null spaces of Macaulay multiplicity matrices  (see for example \cite{DayZen2005,lvz08}), we found  that for $n\geq 4$ the deflated system derived from the Macaulay multiplicity matrix was too large to compute. This is because  the nil-index at the origin is $2^{n-1}$, so the size of the Macaulay multiplicity matrix is  $\;n\cdot{{2^{n-1}+n-1}\choose{n-1}}\times{{2^{n-1}+n}\choose{n}}$.  
{
$$\begin{array}{|c|c|c|c|c|c|c|c|}
\hline
\multicolumn{2}{|c|}{} &\multicolumn{3}{|c|}{\hbox{New approach}} & \multicolumn{3}{|c|}{\hbox{Null space}} \\
\hline
n & \hbox{mult} & \hbox{vars} & \hbox{poly} & \hbox{time} & \hbox{vars} & \hbox{poly} & \hbox{time}\\
\hline
2 & 4 & 5 & 9 & 1.476 &8&17&2.157\\
\hline
3 & 8 & 17 & 31 & 5.596&192&241&208 \\
\hline
4 & 16 & 49 & 100 & 19.698 &7189 &19804&>76000\\
\hline
5 & 32 & 129 & 296 & 73.168&N/A&N/A&N/A \\
\hline
6 & 64 & 321 & 819 & 659.59 &N/A&N/A&N/A\\
\hline
\end{array}$$}

\subsection{Caprasse system}\label{Sec:Caprasse}
\noindent We consider the Caprasse system \cite{Caprasse88,Posso98}:
{
$$
\begin{array}{l}
f(x_1,x_2,x_3,x_4) = 
\left[
\begin{array}{l} 
{x_{{1}}}^{3}x_{{3}}-4\,x_{{1}}{x_{{2}}}^{2}x_{{3}}-4\,{x_{{1}}}^{2}x_{{2}}x_{{4}}-2\,{x_{{2}}}^{3}x_{{4}}-4\,{x_{{1}}}^{2}+\\ ~~~~~~~~~10\,{x_{{2}}}^{2}- 4\,x_{{1}}x_{{3}}+10\,x_{{2}}x_{{4}}-2,\\
x_{{1}}{x_{{3}}}^{3}-4\,x_{{2}}{x_{{3}}}^{2}x_{{4}}-4\,x_{{1}}x_{{3}}{x_{{4}}}^{2}-2\,x_{{2}}{x_{{4}}}^{3}-4\,x_{{1}}x_{{3}}+\\ ~~~~~~~~~10\,x_{{2}}x_{{4}}- 4\,{x_{{3}}}^{2}+10\,{x_{{4}}}^{2}-2,\\
{x_{{2}}}^{2}x_{{3}}+2\,x_{{1}}x_{{2}}x_{{4}}-2\,x_{{1}}-x_{{3}},\\
{x_{{4}}}^{2}x_{{1}}+2\,x_{{2}}x_{{3}}x_{{4}}-2\,x_{{3}}-x_{{1}}
\end{array}\right]
\end{array}
$$
}at the multiplicity $4$ root $\bxi=(2, -\sqrt{-3}, 2, \sqrt{-3})$.

We first consider simply deflating the root.  
Using the approaches of \cite{DayZen2005,HauWam13,lvz06}, one iteration suffices.
For example, using an extrinsic and intrinsic version of \cite{DayZen2005,lvz06},
the resulting system consists of 10 and 8 polynomials, respectively,
and 8 and 6 variables, respectively.
Following \cite{HauWam13}, using all minors results in a system
of 20 polynomials in 4 variables which can be reduced to
a system of 8 polynomials in 4 variables using the $3\times3$ minors
containing a full rank~$2\times2$~submatrix.
The approach of \S~\ref{Sec:Deflation} using an $|\bm i|=1$
step creates a deflated system consisting 
of $6$ polynomials in $4$ variables.  
In fact, since the null space of the Jacobian at the root
is $2$ dimensional, adding two polynomials is necessary and sufficient.

Next, we consider the computation of both the point and multiplicity structure.
Using an intrinsic null space approach via a second order Macaulay
matrix, the resulting system consists of $64$ polynomials in $37$ variables.  
In comparison, 
using the primal 
basis \mbox{$\{1,x_1,x_2$, $x_1x_2\}$}, the approach 
of~\S~\ref{Sec:PointMult} 
 constructs a system
of $30$ polynomials in $19$ variables.

\subsection{Examples with multiple iterations}\label{Sec:MultipleIterations}
\noindent In our last set of examples, we consider simply deflating a root of the last three systems 
from \cite[\S~7]{DayZen2005}
and a system from \cite[\S~1]{Lecerf02}, each of which 
required more than one iteration to deflate.  
These four systems and corresponding points are:
{\small\begin{itemize}
\item[1:] $\{x_1^4 - x_2 x_3 x_4, x_2^4 - x_1 x_3 x_4, x_3^4 - x_1 x_2 x_4, x_4^4 - x_1 x_2 x_3\}$ at $(0,0,0,0)$ with $\mult = 131$ and $o = 10$;
\item[2:] $\{x^4, x^2 y + y^4, z + z^2 - 7x^3 - 8x^2\}$ at $(0,0,-1)$ with $\mult = 16$ and $o = 7$;
\item[3:] $\{14x + 33y - 3\sqrt{5}(x^2 + 4xy + 4y^2 + 2) + \sqrt{7} + x^3 + 6x^2y + 12xy^2 + 8y^3, 41x - 18y - \sqrt{5} + 8x^3 - 12x^2y + 6xy^2 - y^3 + 3\sqrt{7}(4xy - 4x^2 - y^2 - 2)\}$ at $Z_3 \approx (1.5055, 0.36528)$ with $\mult = 5$ and $o = 4$;
\item[4:] $\{2x_1 + 2x_1^2 + 2x_2 + 2x_2^2 + x_3^2 - 1, 
\mbox{$(x_1 + x_2 - x_3 - 1)^3-x_1^3$}, \\
(2x_1^3 + 5x_2^2 + 10x_3 + 5x_3^2 + 5)^3 - 1000 x_1^5\}$ at
$(0,0,-1)$ with $\mult = 18$ and $o = 7$.
\end{itemize}}

We compare using the following four methods:
  (A) intrinsic slicing version of \cite{DayZen2005,lvz06};
  (B) isosingular deflation \cite{HauWam13} via a maximal rank submatrix;
  (C) ``kerneling'' method in \cite{GiuYak13};
  (D) approach of \S~\ref{Sec:Deflation} using an $|\bm i|=1$ step.
We performed these methods without the use of preprocessing and postprocessing
as mentioned in \S~\ref{Sec:Deflation} to directly compare the 
number of nonzero distinct polynomials, variables, and iterations
for each of these four deflation methods.
\vskip -0.05in
{
$$
  \begin{array}{|c|c|c|c|c|c|c|c|c|c|c|c|c|}
\hline
 & 
 \multicolumn{3}{|c|}{\hbox{Method A}} & 
 \multicolumn{3}{|c|}{\hbox{Method B}} &
 \multicolumn{3}{|c|}{\hbox{Method C}} & 
 \multicolumn{3}{|c|}{\hbox{Method D}}\\
\cline{2-13}
& 
\hbox{Poly} & \hbox{Var} & \hbox{It} &
\hbox{Poly} & \hbox{Var} & \hbox{It} &
\hbox{Poly} & \hbox{Var} & \hbox{It} &
\hbox{Poly} & \hbox{Var} & \hbox{It} \\
\hline
1 & 16 & 4 & 2 & 22 & 4 & 2 & 22 & 4 & 2 & 16 & 4 & 2 \\
\hline
2 & 24 & 11 & 3 & 11 & 3 & 2 & 12 & 3 & 2 & 12 & 3 & 3 \\
\hline
3 & 32 & 17 & 4 & 6 & 2 & 4 & 6 & 2 & 4 & 6 & 2 & 4 \\
\hline
4 & 96 & 41 & 5 & 54 & 3 & 5 & 54 & 3 & 5 & 22 & 3 & 5 \\
\hline
\end{array}$$}%
For breadth one singular points as in system 3, methods B, C, and D yield
the same deflated system.
Except for methods B and C on the second system, all four methods required the same number of iterations to deflate the root.  
For the first and third systems, our new approach matched
the best of the other methods and resulted in a 
significantly smaller deflated system for~the~last~one.

\def\cprime{$'$} \def\cprime{$'$} \def\cprime{$'$}
\def\sameauthors{------.~}

\end{document}